\documentclass[11pt]{article}
\usepackage[a4paper,margin=1in]{geometry}
\usepackage{amsmath,amssymb,amsthm,mathtools,bm}
\usepackage{enumitem}
\usepackage{booktabs}
\usepackage{microtype}
\usepackage[pdfborder={0 0 0}]{hyperref}
\usepackage{authblk}

\newtheorem{theorem}{Theorem}[section]
\newtheorem{lemma}[theorem]{Lemma}

\newtheorem{corollary}[theorem]{Corollary}

\theoremstyle{remark}

\numberwithin{equation}{section}

\newcommand{\R}{\mathbb R}

\newcommand{\ignore}[1]{}

\title{Scale-uniform inverse inequalities for scaled kernel spaces}
\author{Davoud Mirzaei}
\affil{Department of Information Technology,\\ Uppsala University,
SE-751 05 Uppsala, Sweden\\
\texttt{davoud.mirzaei@it.uu.se}}

\begin{document}
\maketitle

\begin{abstract}
Inverse inequalities are an important tool in the stability and convergence analysis of kernel approximation methods. In a multiscale
setting, however, the trial space changes with the kernel scale $\delta$, and inverse estimates for a fixed kernel are not sufficient. The constants must remain controlled as $\delta\to0$. In this paper, we study scale-uniform inverse inequalities for spaces generated by scaled positive definite kernels whose native spaces are
Sobolev spaces. We first establish inverse estimates in scale-dependent Sobolev spaces. On bounded domains, this leads to a 
scale-uniform inverse inequality in standard Sobolev norms with $L^2$ as the weaker norm. The estimate requires only that the separation distance of the centers be bounded above by a fixed multiple of the kernel scale. This allows the kernel scale to decrease
more slowly than the separation distance, as is relevant in
multiscale refinement. We then establish more general scale-uniform Bernstein inequalities on
the whole space. Our result covers a broad range of weaker and stronger Sobolev indices under the same relation between the separation distance and the kernel scale. The proof works directly with the
Fourier representation of functions in the kernel space and combines a low--high
frequency decomposition with frame estimates for separated exponential
polynomials. This avoids the additional scale factors that arise from
separate comparisons of scaled and standard Sobolev norms. 


\end{abstract}

\section{Introduction}

Inverse inequalities are a standard tool in approximation theory and in the numerical analysis of partial differential equations.  In contrast to a direct approximation estimate, which measures how well a sufficiently
smooth function can be approximated by a finite-dimensional space, an inverse inequality compares two norms of a function that already belongs to the approximation space.  A typical estimate has the form
$$
\|v\|_{H^\tau} \leq C h^{\alpha-\tau} \|v\|_{H^\alpha},
\quad 0\leq\alpha<\tau,
$$
where $h$ is a parameter describing the resolution of the approximation
space.  Such estimates are classical for polynomial and finite element
spaces. See, for example
\cite{BrennerScott2008}.  They follow from
scaling and norm equivalence on a fixed finite-dimensional reference
space.  Inverse inequalities are used throughout finite element analysis,
for instance in stability estimates, error analysis in stronger norms,
and estimates involving discrete differential operators.

The corresponding theory for kernel and radial basis function approximation spaces is more complicated.  Given a set of centers
$X=\{x_1,\ldots,x_N\}\subset \R^d$, consider
$$ 
V_{\Phi,X} = \operatorname{span} \{\Phi(\cdot-x_j):x_j\in X\},
$$
where $\Phi:\R^d\to\R$ is a positive definite function (shift-invariant kernel). 
Unlike a finite element space, $V_{\Phi,X}$ does not in general admit a
decomposition into fixed finite-dimensional local approximation spaces.
The geometry of the centers and the decay and smoothness properties of
the kernel therefore enter directly into the inverse estimate.  The
separation distance 
$$
q_X = \frac{1}{2}\min_{j\neq k}\|x_j-x_k\|_2
$$
rather than the fill distance, naturally plays the role of the resolution parameter.  One expects inequalities
of Bernstein type in which a stronger Sobolev norm is bounded by a
weaker one multiplied by an appropriate negative power of $q_X$.

An inverse inequality for scattered-data kernel spaces was
established in
\cite{NarcowichWardWendland2006}. For kernels whose Fourier transforms
have algebraic decay, and hence whose native spaces are Sobolev spaces,
they derived Sobolev error estimates together with an $L^2$--Bernstein
inequality. An important ingredient of their analysis is the
construction of a band-limited interpolant with bandwidth proportional
to $q_X^{-1}$. This makes explicit the connection between the
separation of the centers and the highest effective frequencies that
can occur in a kernel approximation space. Band-limited functions and
sampling inequalities have since become useful tools in the analysis
of kernel approximation. See also \cite{NarcowichWard2004,NarcowichWardWendland2005,
SchabackWendland2002}.

The whole-space theory was subsequently developed in several directions.
In \cite{Ward2012} the author obtained $L^p$--Bernstein inequalities for RBF
spaces on $\mathbb R^d$ and used them to establish corresponding
inverse theorems. In this setting, Bessel-potential norms of functions
in an RBF space are controlled by their $L^p$ norms with the expected
dependence on the separation distance.
Related Bernstein and inverse estimates have also been developed for
kernel approximation on spheres and manifolds. Direct and inverse Sobolev
estimates for interpolation by spherical basis functions were developed
in \cite{NarcowichSunWardWendland2007}. In particular, the separation
distance again plays the role of an effective frequency scale.
In
\cite{MhaskarNarcowichPrestinWard2010} 
$L^p$--Bernstein estimates and inverse theorems were further established for spherical basis
function spaces. The analysis in \cite{MhaskarNarcowichPrestinWard2010} makes use of stability estimates
for the coefficients of spherical basis function networks, which is 
related in spirit to the frame estimates used in the present work.
Related frame and sampling constructions on the sphere can be found in
\cite{NarcowichPetrushevWard2006, MhaskarNarcowichPrestinWard2000}.
Inverse inequalities have also been used in the stability analysis of
kernel discretizations of PDEs on the sphere. For example,
\cite{Mirzaei2018-1,Mirzaei2020-1} employ such estimates in the analysis
of Petrov--Galerkin and collocation methods for spherical PDEs.

These results concern fixed kernels. In the multiscale setting considered
here, the kernel itself is replaced by a scaled kernel $\Phi_\delta$,
and the scale $\delta$ varies between levels. The question is therefore
whether the classical dependence on the separation distance can be
retained with constants that remain uniform as $\delta\to0$.

The passage from the whole space, or from a manifold without boundary, to a bounded domain introduces an additional difficulty.  A
whole-space inverse inequality cannot simply be restricted to a bounded domain $\Omega$, since the weaker norm of a kernel function on $\mathbb R^d$ cannot in general be controlled by its restriction to $\Omega$.  This
boundary issue motivated a different line of work based on localized
kernel bases in
\cite{HangelbroekNarcowichRiegerWard2018Manifolds, HangelbroekNarcowichRiegerWard2018}. They
developed direct and inverse estimates on bounded domains using local Lagrange bases.  Their approach provides $L^p$ and Sobolev inverse estimates for spaces generated by localized basis functions, including important classes such as Mat\'ern and surface-spline kernels.  The construction treats the boundary by embedding the physical center set into a suitably enlarged set of centers and constructing localized Lagrange functions using this ambient configuration.  Thus, these results provide a bounded-domain inverse theory, but the approximation spaces and the treatment of the boundary differ from the standard kernel trial space generated solely by translates centered in the physical domain.

For collocation methods for partial differential equations, it is
particularly useful to have an inverse inequality directly for the
original kernel trial space constructed on a set $X\subset \overline \Omega$.
Such a bounded-domain result was established in \cite{CheungLingSchaback2018} as a tool for convergence analysis of least-squares kernel collocation methods. The authors adapted the band-limited interpolation argument to bounded domains and proved an
inverse inequality between Sobolev norms. Their result requires the
weaker Sobolev index to be larger than $d/2$.  This restriction is
natural in their proof because point evaluation and kernel
interpolation are used at that Sobolev level.
The restriction on the weaker Sobolev index has direct consequences
for PDE applications.  In \cite{CheungLingSchaback2018}, the inverse
estimate is an important tool for deriving stability conditions for
overdetermined least-squares Kansa collocation.  For second-order
elliptic problems in dimensions two and three, the admissible Sobolev
range leads to convergence estimates in $H^2(\Omega)$.
The analysis demonstrates clearly how an inverse
inequality can determine the relation between trial and test
discretizations, and hence the amount of oversampling required for
stability.  More generally, inverse estimates are closely connected
with the stability analysis of strong-form kernel collocation and
least-squares discretizations of PDEs. See also the broader discussion
of kernel techniques for meshless methods in
\cite{SchabackWendland2006,Wendland2005}.

A recent result in \cite{SunLing2025} extends
the bounded-domain theory. The authors establish inverse inequalities for
kernel approximation spaces on bounded Lipschitz domains and compact Riemannian manifolds and, in particular, remove the restriction
$\alpha>d/2$ on the weaker Sobolev index in the bounded-domain Bernstein inequality. They use the interpolation theory in Sobolev spaces to prove their results. 

The inverse inequalities discussed above are essentially
\emph{fixed-kernel} results.  This distinction is important for
multiscale kernel methods.  In such methods, the kernel itself changes
with the discretization level.  Starting from a reference kernel
$\Phi$, one introduces
$$
\Phi_\delta(x)=\delta^{-d}\Phi(x/\delta),
$$
where the scale or support parameter $\delta$ decreases as the centers
are refined.  The corresponding trial space is
$$
V_{\Phi_\delta,X}=\operatorname{span}\{\Phi_\delta(\cdot-x_j):x_j\in X\}.
$$
For each fixed $\delta$, an existing inverse theorem may in principle
be applied to $\Phi_\delta$.  This, however, is not sufficient for a
multiscale analysis.  The constant in the resulting estimate may
depend on the scaled kernel and hence on $\delta$.  If that dependence
deteriorates as $\delta\to0$, the estimate cannot be used uniformly
over the levels of the multiscale method.

This issue is not just a matter of keeping track of constants.  The
native norm of $\Phi_\delta$ is naturally equivalent to a
scale-dependent Sobolev norm with Fourier weight
$$
(1+\delta^2\|\omega\|_2^2)^\tau.
$$
Consequently, a direct application of fixed-kernel arguments 
produces estimates in scaled Sobolev norms.  Converting the two sides
separately to standard Sobolev norms may introduce additional powers
of $\delta$. For instance, \cite{LiCao2015}
derives a Bernstein inequality for multiscale interpolation on the
sphere using scaled compactly supported radial basis functions.
The resulting bound contains an explicit term of the
form $\delta^{-\tau}$.
Such factors are harmless when $\delta$ is fixed, but
they become significant when $\delta$ tends to zero under refinement.

The question considered here is different. We seek Bernstein
inequalities in Sobolev norms whose constants remain uniform
as the kernel scale tends to zero and whose final dependence is
expressed solely in terms of the separation distance. In particular,
on $\mathbb R^d$ we establish such estimates under 
condition
$q_X\leq c_0\delta$, without requiring quasi-uniformity of the center
sets and without introducing any additional power of $\delta$ in the
final bound. On a bounded domain $\Omega$, we obtain a corresponding
scale-uniform inverse inequality under the same condition, but presently
only at the $L^2(\Omega)$ that is when the weaker Sobolev
index is $\alpha=0$.

\subsection{Kernels}

Let $\Phi:\mathbb R^d\to\mathbb R$ be a continuous strictly positive
definite kernel satisfying
\begin{equation}
\label{eq:fourier_decay}
c_\Phi(1+\|\omega\|_2^2)^{-\sigma}
\le
\widehat\Phi(\omega)
\le
C_\Phi(1+\|\omega\|_2^2)^{-\sigma},
\quad
\omega\in\mathbb R^d,
\end{equation}
for some $\sigma>d/2$ and constants
$0<c_\Phi\le C_\Phi<\infty$.  
Here, $\widehat \Phi$ denotes the Fourier transform of $\Phi$, defined by
$$
\widehat \Phi(\omega) = (2\pi)^{-d/2}\int_{\R^d}\Phi(x)e^{-ix^T\omega}dx.
$$
Under this assumption, the native space
of $\Phi$ is norm-equivalent to $H^\sigma(\mathbb R^d)$; see
\cite{Wendland2005,SchabackWendland2006}. The native space norm is defined by
$$
\|u\|_\Phi^2 :=\int_{\R^d} \frac{|\widehat u(\omega)|^2}{\widehat \Phi(\omega)}\, d\omega,
$$
while the Sobolev norm on $H^\sigma(\R^d)$ is defined, using the Fourier transform, by
$$
\|u\|_{H^\sigma(\R^d)}^2 :=(2\pi)^{-d/2}\int_{\R^d} |\widehat u(\omega)|^2(1+\|\omega\|_2^2)^\sigma\, d\omega.
$$
We also consider the scaled
kernel
$$
\Phi_\delta(x)=\delta^{-d}\Phi(x/\delta),
\quad
0<\delta\le1,
$$
whose Fourier transform is
\begin{equation*}
\widehat{\Phi_\delta}(\omega) =\widehat{\Phi}(\delta\omega).
\end{equation*}
For a finite set $X\subset\overline\Omega$, let
$$
V_{\Phi_\delta,X}=\operatorname{span}
\{\Phi_\delta(\cdot-x_j):x_j\in X\}.
$$
Throughout this paper we assume that the reference kernel $\Phi$ satisfies the decay condition 
\eqref{eq:fourier_decay}
for some $\sigma>d/2$ and constants $0<c_\Phi\leq C_\Phi$.
Then for $0<\delta\leq1$, 
\begin{equation}\label{eq:scaled-fourier-decay}
c_\Phi(1+\delta^2\|\omega\|_2^2)^{-\sigma}
\leq
\widehat{\Phi_\delta}(\omega)
\leq
C_\Phi(1+\delta^2\|\omega\|_2^2)^{-\sigma}.
\end{equation}
We denote the fill distance, separation distance, and mesh ratio of a finite point set $X$ by
$$
h_{X,\Omega}
=\sup_{x\in\Omega}\min_{x_j\in X}\|x-x_j\|_2,
\quad q_X
=\frac12\min_{x_j\ne x_k}\|x_j-x_k\|_2, \quad \rho_X =\frac{h_{X,\Omega}}{q_X},
$$
respectively.  
%When the ambient domain is clear, we simply write $h_X$.  
A family of point sets is called quasi-uniform if the mesh
ratios remain uniformly bounded.
Our aim is to understand when inverse estimates for
$V_{\Phi_\delta,X}$ can be obtained with constants independent of
$\delta$, $X$, and the dimension of the trial space.  Of particular
interest is the regime
$$
q_X\leq c_0\delta
$$
which permits the kernel scale to decrease more slowly than the
separation distance and is therefore appropriate for multiscale
refinement.

\subsection{Main contributions}

We first work with Sobolev norms adapted to the scale $\delta$.  In
these norms the native-space structure of $\Phi_\delta$ is uniform in
$\delta$, which allows the classical band-limited interpolation
argument of \cite{NarcowichWardWendland2006,
CheungLingSchaback2018} to be carried over to scaled kernel spaces.
Using interpolation, we obtain scale-uniform estimates over a broad
range of Sobolev indices.  An immediate consequence is a
bounded-domain endpoint estimate with $L^2(\Omega)$ on the weaker side
under the condition $q_X\leq c_0\delta$.

We then consider standard Sobolev norms on the whole space.  A direct
conversion from the scaled estimate is not sharp. It introduces an
additional negative power of $\delta$.  To avoid this loss, we work
directly with the Fourier representation of a kernel network.  The
Fourier integral is separated into low and high frequency regions at
the natural frequency scale $q_X^{-1}$.  Frame estimates for
exponential polynomials associated with separated centers provide the
coefficient control required in the high frequency region.  Comparing
the two Sobolev orders before estimating the scale-dependent kernel
multiplier allows the powers of $\delta$ to cancel.  This yields a
scale-uniform whole-space Bernstein inequality with the expected power
of the separation distance under the relatively mild condition
$q_X\leq c_0\delta$.

This paper complements the existing fixed-kernel inverse theory in
\cite{MhaskarNarcowichPrestinWard2010, NarcowichWardWendland2006,Ward2012, HangelbroekNarcowichRiegerWard2018,
CheungLingSchaback2018,SunLing2025}
by making the dependence on the kernel scale explicit and establishing
estimates that remain uniform as the scale parameter tends to zero.
It leaves open the problem of proving a general scale-uniform inverse inequality on bounded domains for the full range of Sobolev indices.  

The remainder of the paper is organized as follows.
Section~\ref{sec:scale-uniform-bernstein} develops the inverse theory
in scaled Sobolev spaces. In Section~\ref{sect:inverse_Htau_L2}, we use
these results to derive a scale-uniform inverse inequality in standard
Sobolev norms on a bounded domain $\Omega$, with $L^2(\Omega)$ as the
weaker norm, corresponding to the case $\alpha=0$.
In Section~\ref{sect:general_inverse_Rd}, we turn to the whole space
$\mathbb R^d$ and establish a general scale-uniform inverse inequality
which covers the full range of admissible Sobolev indices.
The paper concludes with a discussion of the main results and directions
for future research.

\section{Scale-uniform inverse inequalities in scaled Sobolev spaces}\label{sec:scale-uniform-bernstein}

In this section, we derive an inverse inequality for functions in $V_{\Phi_\delta,X}$ in terms of the corresponding scaled Sobolev norms. Although these scaled norms may not be the natural norms for the applications, the analysis provides useful insight into the role of the kernel scale $\delta$ and we finally use it to establish an inverse inequality in standard Sobolev norms in Section \ref{sect:inverse_Htau_L2}. The proof will use the standard band-limited interpolation technique commonly used to establish inverse inequalities for fixed-scale kernel spaces
\cite{CheungLingSchaback2018,NarcowichWardWendland2006}.

For $s\geq0$, define the scaled Sobolev norm on $\mathbb{R}^d$ by
\begin{equation*}
\|v\|_{H_\delta^s(\mathbb{R}^d)}^2:=(2\pi)^{-d/2}\int_{\mathbb{R}^d}
(1+\delta^2\|\omega\|_2^2)^s |\widehat v(\omega)|^2 \,d\omega.
\end{equation*}
We define the scaled Sobolev space on $\Omega$ by
restriction,
$$
H_\delta^s(\Omega):=\left\{ v=w|_\Omega:\; w\in H_\delta^s(\mathbb R^d) \right\},
$$
equipped with the quotient norm
\begin{equation}\label{eq:scaled-sobolev-domain}
\|v\|_{H_\delta^s(\Omega)}:=\inf_{\substack{w\in H_\delta^s(\mathbb R^d)\\
    w|_\Omega=v}}\|w\|_{H_\delta^s(\mathbb R^d)}.
\end{equation}
Thus $H_\delta^s(\Omega)$ is the restriction (quotient) space
associated with $H_\delta^s(\mathbb R^d)$.  
This is the usual
restriction-space definition of Sobolev/Bessel-potential spaces on
domains, here applied to the scaled Fourier weight
$(1+\delta^2\|\omega\|_2^2)^s$. See for example
\cite{Rychkov1999}. In particular, for $s=0$ the scaled space is independent of $\delta$
and
$$
H_\delta^0(\Omega)=L^2(\Omega)
$$
isometrically.
By the definition of the quotient norm, the restriction operator
$$
R_\Omega:H_\delta^s(\mathbb R^d)\to H_\delta^s(\Omega), \quad R_\Omega w=w|_\Omega,
$$
has norm at most one.  In particular,
\begin{equation*}
\|w|_\Omega\|_{H_\delta^s(\Omega)} \le \|w\|_{H_\delta^s(\mathbb R^d)}.
\end{equation*}
From \eqref{eq:scaled-fourier-decay}, the native space norm of
$\Phi_\delta$ is uniformly equivalent to the scaled Sobolev norm, i.e.,
\begin{equation}\label{eq:uniform-native-equivalence}
(2\pi)^{-d/4}c_\Phi^{\frac12}\|v\|_{\Phi_\delta}
\leq \|v\|_{H_\delta^\sigma(\mathbb{R}^d)}
\leq (2\pi)^{-d/4}C_\Phi^{\frac12}\|v\|_{\Phi_\delta},
\end{equation}
where $c_\Phi$ and $C_\Phi$ are the same constants as in \eqref{eq:fourier_decay} and \eqref{eq:scaled-fourier-decay} which depend only on fixed reference kernel $\Phi$ but and are
independent of $\delta$.

We first formulate a scaled version of the band-limited interpolation
result. The following lemma is inspired by a similar result in
\cite{Wendland2010Multiscale}. The main difference is that, in our
setting, the Sobolev norm on the right-hand side is taken over the
bounded domain $\Omega$, rather than over the whole space $\mathbb{R}^d$.

\begin{lemma}
\label{lem:scaled-bandlimited}
Let $\alpha>d/2$.  Then there exist constants
$\kappa_{\alpha,d}>0$ and $C_{\alpha,d,\Omega}>0$, independent of
$\delta$, $X$, and $v$, such that for every
$v\in H_\delta^\alpha(\Omega)$ there exists
$f_{\rho,\delta}\in H_\delta^\alpha(\mathbb{R}^d)$ satisfying
\begin{equation*}
f_{\rho,\delta}|_X=v|_X,
\end{equation*}
\begin{equation*}
\operatorname{supp}\widehat{f_{\rho,\delta}}
\subset
B(0,\rho),
\quad
\rho=\kappa_{\alpha,d}q_X^{-1},
\end{equation*}
and
\begin{equation*}
\label{eq:scaled-bandlimited-bound}
\|f_{\rho,\delta}\|_{H_\delta^\alpha(\mathbb{R}^d)}
\leq
C_{\alpha,d,\Omega}
\|v\|_{H_\delta^\alpha(\Omega)}.
\end{equation*}
\end{lemma}

\begin{proof}
Let $\varepsilon>0$.  By the definition of the restriction norm in \eqref{eq:scaled-sobolev-domain},
there exists an extension
$w\in H_\delta^\alpha(\mathbb{R}^d)$ of $v$ such that
\begin{equation}
\label{eq:near-minimal-extension}
\|w\|_{H_\delta^\alpha(\mathbb{R}^d)}
\leq
(1+\varepsilon)
\|v\|_{H_\delta^\alpha(\Omega)}.
\end{equation}
Define
$$
\widetilde w(y)=\delta^{d/2}w(\delta y),
\quad
\widetilde X=\delta^{-1}X.
$$
Then
$q_{\widetilde X}={q_X}/{\delta}$.
Since
$$
\widehat{\widetilde w}(\xi)
=
\delta^{-d/2}\widehat w(\xi/\delta),
$$
a change of variables gives
\begin{equation}
\label{eq:scaled-norm-change}
\|\widetilde w\|_{H^\alpha(\mathbb{R}^d)} =\|w\|_{H_\delta^\alpha(\mathbb{R}^d)}.
\end{equation}
Apply the standard band-limited interpolation result to
$\widetilde w$ on the set $\widetilde X$ (See \cite{NarcowichWardWendland2006}). According to this result, there exists a function
$\widetilde w_\rho$ such that
$$
\widetilde w_\rho|_{\widetilde X}
= \widetilde w|_{\widetilde X},
\quad  \operatorname{supp}\widehat{\widetilde w_\rho}
\subset B\left(0,\widetilde \rho\,\right),\quad \widetilde \rho = \kappa_{\alpha,d}\,q_{\widetilde X}^{-1},
$$
and
$$
\|\widetilde w_\rho\|_{H^\alpha(\mathbb{R}^d)}
\leq
C_{\alpha,d}
\|\widetilde w\|_{H^\alpha(\mathbb{R}^d)}.
$$
Set
$$
f_{\rho,\delta}(x) = \delta^{-d/2}\widetilde w_\rho(x/\delta).
$$
Then
$f_{\rho,\delta}|_X=v|_X$.  Moreover,
$\widehat{f_{\rho,\delta}}(\omega)= \delta^{d/2}\widehat{\widetilde w_\rho}(\delta\omega)$,
and therefore
$$
\operatorname{supp}\widehat{f_{\rho,\delta}}
\subset B\left(0, \delta^{-1}\widetilde \rho\,\right) =B(0,\rho).
$$
Finally,
\begin{align*}
\|f_{\rho,\delta}\|_{H_\delta^\alpha(\mathbb{R}^d)}^2 &= \int_{\mathbb{R}^d} (1+\delta^2\|\omega\|_2^2)^\alpha |\widehat{f_{\rho,\delta}}(\omega)|^2 \,d\omega \\
& = \int_{\mathbb{R}^d} (1+\|\delta\omega\|_2^2)^\alpha |\widehat{w_{\rho}}(\delta\omega)|^2 \,d(\delta\omega)
= \|\widetilde w_\rho\|_{H^\alpha(\mathbb{R}^d)}.
\end{align*}
Together with \eqref{eq:scaled-norm-change}, this gives
$$
\|f_{\rho,\delta}\|_{H_\delta^\alpha(\mathbb{R}^d)}
\leq
C_{\alpha,d}
\|w\|_{H_\delta^\alpha(\mathbb{R}^d)}.
$$
Using \eqref{eq:near-minimal-extension} and letting
$\varepsilon\to0$ proves the result.
\end{proof}

We next obtain an inverse estimate in the scaled Sobolev
norm.
\begin{theorem}\label{thm:scaled-high-order-inverse}
Let $d/2<\alpha\leq\sigma$.
Fix a constant $c_0>0$.  Then there exists $C=C_{c_0,\alpha,d,\Phi,\Omega}>0$, independent of $\delta$ and $X$, such that
\begin{equation}\label{eq:scaled-high-order-inverse}
\|v\|_{H_\delta^\sigma(\Omega)}
\leq C\left(\frac{\delta}{q_X}\right)^{\sigma-\alpha}
\|v\|_{H_\delta^\alpha(\Omega)}
\end{equation}
for every $v\in V_{\Phi_\delta,X}$ satisfying $q_X\leq c_0\delta$.
\end{theorem}

\begin{proof}
Let $v\in V_{\Phi_\delta,X}$ and let
$f_{\rho,\delta}$ be the band-limited function from
Lemma~\ref{lem:scaled-bandlimited}.  Since
$f_{\rho,\delta}$ and $v$ agree on $X$ and
$v\in V_{\Phi_\delta,X}$, the function $v$ is the
$\Phi_\delta$-interpolant of $f_{\rho,\delta}$ on $X$.  Hence the
minimum native norm property gives
$
\|v\|_{\Phi_\delta}\leq\|f_{\rho,\delta}\|_{\Phi_\delta}.
$
Using \eqref{eq:uniform-native-equivalence},
\begin{equation}\label{eq:inverse-native-chain}
\begin{aligned}
\|v\|_{H_\delta^\sigma(\Omega)}
&\leq \|v\|_{H_\delta^\sigma(\mathbb{R}^d)} \\
&\leq (2\pi)^{-d/4}C_\Phi^{\frac12}\|v\|_{\Phi_\delta}\\
&\leq (2\pi)^{-d/4}C_\Phi^{\frac12}\|f_{\rho,\delta}\|_{\Phi_\delta}\\
&\leq c_\Phi^{-\frac12} C_\Phi^{\frac12}\|f_{\rho,\delta}\|_{H_\delta^\sigma(\mathbb{R}^d)}.
\end{aligned}
\end{equation}
Since
$$
\operatorname{supp}\widehat{f_{\rho,\delta}}
\subset B(0,\rho),\quad \rho=\kappa_{\alpha,d}q_X^{-1},
$$
we have
\begin{align*}
\|f_{\rho,\delta}\|_{H_\delta^\sigma(\mathbb{R}^d)}^2
&=(2\pi)^{-d/2}\int_{\|\omega\|_2\leq\rho}(1+\delta^2\|\omega\|_2^2)^\sigma
|\widehat{f_{\rho,\delta}}(\omega)|^2\,d\omega \\
&=(2\pi)^{-d/2}\int_{\|\omega\|_2\leq\rho}(1+\delta^2\|\omega\|_2^2)^{\sigma-\alpha}(1+\delta^2\|\omega\|_2^2)^\alpha
|\widehat{f_{\rho,\delta}}(\omega)|^2\,d\omega \\
&\leq (2\pi)^{-d/2}
(1+\delta^2\rho^2)^{\sigma-\alpha}
\|f_{\rho,\delta}\|_{H_\delta^\alpha(\mathbb{R}^d)}^2.
\label{eq:bandlimited-scaled-bernstein}
\end{align*}
Since $q_X/\delta\leq c_0$, we can write 
$$
1+\delta^2\rho^2
=1+\kappa_{\alpha,d}^2\frac{\delta^2}{q_X^2} = \left(\frac{q_X^2}{\delta^2}+\kappa_{\alpha,d}^2\right)\frac{\delta^2}{q_X^2}
\leq
\left(c_0^2+\kappa_{\alpha,d}^2\right)
\frac{\delta^2}{q_X^2}.
$$
Thus
$$
\|f_{\rho,\delta}\|_{H_\delta^\sigma(\mathbb{R}^d)}
\leq
C
\left(
\frac{\delta}{q_X}
\right)^{\sigma-\alpha}
\|f_{\rho,\delta}\|_{H_\delta^\alpha(\mathbb{R}^d)},
$$
where $C$ depends only on $\Phi$, $c_0$, $\alpha$, and $d$. 
Using Lemma~\ref{lem:scaled-bandlimited} and
\eqref{eq:inverse-native-chain} proves
\eqref{eq:scaled-high-order-inverse}.
\end{proof}

The rest of this section is devoted to extend the range of the stronger Sobolev norm in Theorem \ref{thm:scaled-high-order-inverse} from the only end index $\sigma>d/2$ to any index $\tau$ with $0\leq\tau\leq\sigma$, and to extend the range of the weaker Sobolev norm form $\alpha>d/2$ to $\alpha\geq 0$.  
For this aim we require interpolation in the scaled Sobolev scale. 

\begin{lemma}
\label{lem:scaled-interpolation}
Let $\Omega\subset\mathbb R^d$ be a bounded Lipschitz domain and let
$0\le t<\alpha<\sigma$ such that
$\alpha=(1-\theta)t+\theta\sigma$ for $0<\theta<1$.
Then there exists a constant $C=C_{\Omega,t,\alpha,\sigma}>0$,
independent of $\delta$, such that
\begin{equation}
\label{eq:scaled-interpolation}
\|v\|_{H_\delta^\alpha(\Omega)}
\le
C\|v\|_{H_\delta^t(\Omega)}^{1-\theta}\|v\|_{H_\delta^\sigma(\Omega)}^\theta
\end{equation}
for every $v\in H_\delta^\sigma(\Omega)$ and every $0<\delta\le1$.
\end{lemma}

\begin{proof}
We first consider the corresponding spaces on $\mathbb R^d$.
By definition,
$$
\|w\|_{H_\delta^s(\mathbb R^d)}^2
=\int_{\mathbb R^d}(1+\delta^2\|\omega\|_2^2)^s |\widehat w(\omega)|^2\,d\omega.
$$
Since
$\alpha=(1-\theta)t+\theta\sigma$,
H\"older's inequality gives
\begin{align*}
\|w\|_{H_\delta^\alpha(\mathbb R^d)}^2
&= (2\pi)^{-d/2}
\int_{\mathbb R^d}
\left[
(1+\delta^2\|\omega\|_2^2)^t
|\widehat w(\omega)|^2
\right]^{1-\theta}
\left[ (1+\delta^2\|\omega\|_2^2)^\sigma |\widehat w(\omega)|^2\right]^\theta\,d\omega
\\
&\leq
\|w\|_{H_\delta^t(\mathbb R^d)}^{2(1-\theta)}
\|w\|_{H_\delta^\sigma(\mathbb R^d)}^{2\theta}.
\end{align*}
Hence
\begin{equation*}
\|w\|_{H_\delta^\alpha(\mathbb R^d)}
\leq
\|w\|_{H_\delta^t(\mathbb R^d)}^{1-\theta} \|w\|_{H_\delta^\sigma(\mathbb R^d)}^\theta.
\end{equation*}
Thus the whole-space interpolation constant is one for every
$\delta>0$.
We now pass to the restriction spaces.  Set
$\Omega_\delta:=\delta^{-1}\Omega$
and define the dilation
$$
(T_\delta w)(y) := \delta^{d/2}w(\delta y).
$$
Its Fourier transform is
$$
\widehat{T_\delta w}(\xi) =\delta^{-d/2}\widehat w(\xi/\delta).
$$
Therefore, after the change of variables $\xi=\delta\omega$,
$$
\begin{aligned}
\|T_\delta w\|_{H^s(\mathbb R^d)}^2
&=(2\pi)^{-d/2} \int_{\mathbb R^d} (1+\|\xi\|_2^2)^s \delta^{-d} |\widehat w(\xi/\delta)|^2\,d\xi
\\
&=(2\pi)^{-d/2}\int_{\mathbb R^d} (1+\delta^2\|\omega\|_2^2)^s |\widehat w(\omega)|^2\,d\omega
\\
&= \|w\|_{H_\delta^s(\mathbb R^d)}^2.
\end{aligned}
$$
Thus $T_\delta$ is an isometry from
$H_\delta^s(\mathbb R^d)$ onto $H^s(\mathbb R^d)$.
The same is true for the corresponding restriction spaces.
For $v\in H_\delta^s(\Omega)$, define
$$
v_\delta(y):=\delta^{d/2}v(\delta y),
\quad y\in\Omega_\delta.
$$
Using the definition of the domain norms in \eqref{eq:scaled-sobolev-domain} and the
one-to-one correspondence
$$
w|_\Omega=v
\quad\Longleftrightarrow\quad
(T_\delta w)|_{\Omega_\delta}=v_\delta,
$$
we obtain
\begin{equation}
\label{eq:scaled-domain-dilation-isometry}
\|v\|_{H_\delta^s(\Omega)} =\|v_\delta\|_{H^s(\Omega_\delta)}.
\end{equation}
\ignore{
Since $\Omega$ is Lipschitz, every dilated domain
$\Omega_\delta=\delta^{-1}\Omega$ is Lipschitz with the same
Lipschitz constants.  Moreover, for $0<\delta\le1$, the radii of
the corresponding boundary charts do not decrease.  Hence the
family $\{\Omega_\delta:0<\delta\le1\}$ has uniformly bounded
Lipschitz character.  The standard interpolation inequality for
Bessel-potential spaces on Lipschitz domains therefore holds
uniformly on this family:
$$
\|z\|_{H^\alpha(\Omega_\delta)}
\le
C \|z\|_{H^t(\Omega_\delta)}^{1-\theta} \|z\|_{H^\sigma(\Omega_\delta)}^\theta,
$$
where $C$ depends only on the Lipschitz character of $\Omega$ and
on $t,\alpha,\sigma$, and is independent of $\delta$. See for
example
\cite{BerghLofstrom1976,Triebel1978,Rychkov1999}.}
Since $\Omega$ is Lipschitz, its boundary can be covered by finitely
many coordinate patches of radii $r_j>0$ in which $\partial\Omega$ is
represented by Lipschitz graphs with Lipschitz constants bounded by
some $M>0$. Under the dilation $x\mapsto \delta^{-1}x$, the
corresponding graph functions have the same Lipschitz constants,
whereas the chart radii become $r_j/\delta\geq r_j$ for
$0<\delta\leq 1$. Thus, the family
$
\{\Omega_\delta=\delta^{-1}\Omega:0<\delta\leq 1\}
$
has uniform local Lipschitz character.
Consequently, the extension operators for Bessel-potential spaces on
Lipschitz domains may be chosen with operator norms bounded uniformly
for $0<\delta\leq 1$; see, for example, Rychkov's universal extension
theorem for Besov and Triebel--Lizorkin spaces
\cite{Rychkov1999}, noting that $H^s=F^s_{2,2}$. Combining these
uniformly bounded extension operators with the whole-space
interpolation identity
$$
[H^t(\mathbb{R}^d),H^\sigma(\mathbb{R}^d)]_\theta
=
H^\alpha(\mathbb{R}^d),
\qquad
\alpha=(1-\theta)t+\theta\sigma,
$$
gives
$$
\|z\|_{H^\alpha(\Omega_\delta)}
\leq
C
\|z\|_{H^t(\Omega_\delta)}^{1-\theta}
\|z\|_{H^\sigma(\Omega_\delta)}^\theta,
$$
where $C>0$ depends only on $t,\alpha,\sigma$ and the local Lipschitz
character of $\Omega$, and is independent of
$0<\delta\leq 1$. cf.
\cite{BerghLofstrom1976,Triebel1978,Rychkov1999}.
Applying this estimate to $z=v_\delta$ and using
\eqref{eq:scaled-domain-dilation-isometry} for the three Sobolev
indices gives \eqref{eq:scaled-interpolation}
with $C$ independent of $\delta$.
\end{proof}

Now, we apply the above interpolation theory on inverse inequality \eqref{eq:scaled-high-order-inverse} to extend the range of the stronger Sobolev norm from the only end index $\sigma>d/2$ to any index $\tau$ with $0\leq \tau\leq\sigma$, and to extend the range of the weaker Sobolev norm form $\alpha>d/2$ to $\alpha\geq 0$.  

\begin{theorem}\label{thm:scaled-general-inverse}
Let $\Omega\subset\mathbb{R}^d$ be a bounded Lipschitz domain and let $\Phi$ satisfy the decay condition \eqref{eq:fourier_decay}
with $ \sigma>d/2$. Assume that
$0\leq\alpha\leq \tau\leq\sigma$.
Fix $c_0>0$. Then, there exists a constant
$C=C_{c_0,\alpha,\tau,\sigma,d,\Phi,\Omega}>0$
independent of $\delta$ and $X$, such that
\begin{equation}\label{eq:scaled-general-inverse}
\|v\|_{H_\delta^\tau(\Omega)}
\leq C \left(\frac{\delta}{q_X}\right)^{\tau-\alpha} \|v\|_{H_\delta^\alpha(\Omega)}
\end{equation}
for every $v\in V_{\Phi_\delta,X}$ satisfying
$q_X\leq c_0\delta$ where $0<\delta\leq1$.
\end{theorem}

\begin{proof}
First we extend the range of $\alpha$. Choose
$$
\beta\in
\left(
\max\left\{\alpha,\frac d2\right\},\sigma
\right).
$$
Since $d/2<\beta< \sigma$, Theorem~\ref{thm:scaled-high-order-inverse} gives
\begin{equation}
\label{eq:general-inverse-proof-1}
\|v\|_{H_\delta^\sigma(\Omega)}
\leq
C
\left(\frac{\delta}{q_X}\right)^{\sigma-\beta}
\|v\|_{H_\delta^\beta(\Omega)},
\end{equation}
where $C=C_{c_0,\beta,d,\Phi,\Omega}$.
Let
$\theta={(\beta-\alpha)}/{(\sigma-\alpha)}$,
so that
$
\beta=(1-\theta)\alpha+\theta\sigma.
$
By Lemma~\ref{lem:scaled-interpolation},
$$
\|v\|_{H_\delta^\beta(\Omega)}
\leq C
\|v\|_{H_\delta^\alpha(\Omega)}^{1-\theta}
\|v\|_{H_\delta^\sigma(\Omega)}^\theta,
$$
where $C=C_{\alpha,\sigma,\beta,\Omega}$.
Substituting this estimate into
\eqref{eq:general-inverse-proof-1} gives
$$
\|v\|_{H_\delta^\sigma(\Omega)}^{1-\theta}
\leq
C
\left(\frac{\delta}{q_X}\right)^{\sigma-\beta}
\|v\|_{H_\delta^\alpha(\Omega)}^{1-\theta},
$$
where $C=C_{c_0,\beta,\alpha,d,\Phi,\Omega}$
Since
$1-\theta = {(\sigma-\beta)}/{(\sigma-\alpha)}$,
we obtain
\begin{equation}\label{eq:sigma-t-inverse}
\|v\|_{H_\delta^\sigma(\Omega)}
\leq
C
\left(\frac{\delta}{q_X}\right)^{\sigma-\alpha}
\|v\|_{H_\delta^\alpha(\Omega)}.
\end{equation}
Since the argument holds for every index $\beta$ in the mentioned interval, we may fix any such $\beta$ once and for all, remove the explicit dependence of the constant $C$ on $\beta$ and keep it dependent on $d$, $\tau$ and $\alpha$ instead (together with the other fixed problem parameters). Next, let $\eta={(\tau-\alpha)}/{(\sigma-\alpha)}$,
so that
$\tau=(1-\eta)\alpha+\eta\sigma$.
Another application of Lemma~\ref{lem:scaled-interpolation} yields
$$
\|v\|_{H_\delta^\tau(\Omega)}
\leq
C
\|v\|_{H_\delta^\alpha(\Omega)}^{1-\eta}
\|v\|_{H_\delta^\sigma(\Omega)}^\eta,
$$
where $C=C_{\alpha,\tau,\sigma,\Omega}$.
Using \eqref{eq:sigma-t-inverse}, we can write 
$$
\begin{aligned}
\|v\|_{H_\delta^\tau(\Omega)}
&\leq
C
\|v\|_{H_\delta^\alpha(\Omega)}^{1-\eta}
\left[
\left(\frac{\delta}{q_X}\right)^{\sigma-\alpha}
\|v\|_{H_\delta^\alpha(\Omega)}
\right]^\eta
\\
&=
C
\left(\frac{\delta}{q_X}\right)^{\eta(\sigma-\alpha)}
\|v\|_{H_\delta^\alpha(\Omega)}.
\end{aligned}
$$
where $C=C_{c_0,\alpha,\tau,\sigma,d,\Phi,\Omega}$.
The fact that $\eta(\sigma-\alpha)=\tau-\alpha$,
proves \eqref{eq:scaled-general-inverse}.
\end{proof}

\section{From scaled norms to standard norms}\label{sect:inverse_Htau_L2}

Both sides of the inverse inequality \eqref{eq:scaled-general-inverse} are expressed with the scaled Sobolev norms $H_\delta^\tau(\Omega)$ and $H_\delta^\alpha(\Omega)$. In application, however, we mostly need the measurement in standard Sobolev norms.  
The following corollary proves a special case of inverse inequalities on standard Sobolev norm for which the weaker norm is only the $L^2(\Omega)$ norm.

\begin{corollary} \label{cor:scale-uniform-bernstein}
Let $\Omega\subset\mathbb{R}^d$ be a bounded Lipschitz domain and let
$\Phi$ satisfy \eqref{eq:fourier_decay} for some
$\sigma>d/2$.  Assume that $0\leq\tau\leq\sigma$. Fix $c_0>0$. Then there exists a constant
$
C=C_{d,\sigma,\tau,\Phi,\Omega,c_0}>0
$
such that, for every $0<\delta\leq1$, every finite set
$X\subset\Omega$ satisfying $q_X\leq c_0\delta$,
and every $v\in V_{\Phi_\delta,X}$,
\begin{equation} \label{eq:scale-uniform-bernstein-final}
\|v\|_{H^\tau(\Omega)}
\leq
C\, q_X^{-\tau}\|v\|_{L^2(\Omega)}.
\end{equation}
In particular, the constant $C$ is independent of both $\delta$
and $X$.
\end{corollary}

\begin{proof}
Given $v\in H^{\tau}(\Omega)$ and given $\varepsilon>0$, by the definition of the restriction norm in \eqref{eq:scaled-sobolev-domain}, there exists an extension
$w\in H_\delta^\tau(\mathbb{R}^d)$ of $v$ such that
\begin{equation}
\label{eq:near-minimal-extension2}
\|w\|_{H_\delta^\tau(\mathbb{R}^d)}
\leq
(1+\varepsilon)
\|v\|_{H_\delta^\tau(\Omega)}.
\end{equation}
We also have,
$$
1+\|\omega\|_2^2
\leq
\delta^{-2}(1+\delta^2\|\omega\|_2^2).
$$
This gives 
\begin{align*}
    \|w\|_{H^{\tau}(\R^d)}^2 &= C_d \int_{\R^d} (1+\|\omega\|_2^2)^\tau |\widehat w(\omega)|^2 \, d\omega \\
    & \leq C_d \delta^{-2\tau}\int_{\R^d} (1+\delta^2\|\omega\|_2^2)^\tau |\widehat w(\omega)|^2 \, d\omega \\
    & = C_d \delta^{-2\tau} \|w\|_{H_\delta^\tau(\R^d)}^2,
\end{align*}
which shows that 
\begin{equation}\label{eq:bound-on-H-tau-delta}
    \|w\|_{H^{\tau}(\R^d)} \leq C_d \delta^{-\tau} \|w\|_{H_\delta^\tau(\R^d)}.
\end{equation}
Combining \eqref{eq:near-minimal-extension2} and \eqref{eq:bound-on-H-tau-delta} gives 
$$
\|w\|_{H^{\tau}(\R^d)} \leq (1+\varepsilon)C_d \delta^{-\tau} \|v\|_{H_\delta^\tau(\Omega)}.
$$
Now we apply the inverse inequality \eqref{eq:scaled-general-inverse} to obtain 
$$
\|v\|_{H^{\tau}(\Omega)}\leq \|w\|_{H^{\tau}(\R^d)} \leq
(1+\varepsilon) C \delta^{-\tau} \left(\frac{\delta}{q_X}\right)^{\tau-\alpha} \|v\|_{H_\delta^\alpha(\Omega)}
$$
for any $0\leq\alpha\leq \tau$, where $C=C_{c_0,\alpha,\tau,\sigma,d,\Phi,\Omega}$. Letting $\varepsilon$ goes to zero and the fact that $\|v\|_{H^{\tau}(\Omega)}\leq \|w\|_{H^{\tau}(\R^d)}$ result in 
$$
\|v\|_{H^{\tau}(\Omega)}\leq 
 C \delta^{-\alpha}{q_X}^{-\tau} \|v\|_{H_\delta^\alpha(\Omega)}.
$$
Here is the place that we can go only for special case $\alpha=0$ where in this case we have $\delta^{-\alpha}=1$ and
$H_\delta^0(\Omega) = L^2(\Omega)$.
This proves \eqref{eq:scale-uniform-bernstein-final}.
\end{proof}

\section{A general scale-uniform inverse inequality on $\R^d$}\label{sect:general_inverse_Rd}

The inverse estimate established in the previous section uses 
$L^2$
as the weaker norm. A scale-uniform extension to intermediate Sobolev norms $H^\alpha$, $0<\alpha<\tau$, even if we replace the bounded domain $\Omega$ with the whole space $\R^d$, appears to require additional arguments beyond the scaled native-space technique used there. 
The reason is that a straightforward application
of the scaled Sobolev inverse estimate
$$
\|v\|_{H_\delta^\sigma}
\leq C
\left(\frac{\delta}{q_X}\right)^{\sigma-\alpha}
\|v\|_{H_\delta^\alpha}
$$
does not immediately give the desired estimate in the standard
Sobolev norms.  Indeed, using the norm comparisons
$$
\|v\|_{H^\sigma}
\leq C
\delta^{-\sigma}\|v\|_{H_\delta^\sigma},
\quad
\|v\|_{H_\delta^\alpha}
\leq C
\|v\|_{H^\alpha},
$$
leads only to
$$
\|v\|_{H^\sigma}
\leq C
\delta^{-\alpha}q_X^{\alpha-\sigma}
\|v\|_{H^\alpha}.
$$
Thus an additional factor $\delta^{-\alpha}$ is introduced.  This
loss is not caused by the kernel space itself, but by applying two
global norm comparisons separately.  The first comparison attains
its worst scaling in the high frequency regime, whereas the second
one is sharp in the low frequency regime.  Combining their worst
cases therefore loses information about the actual frequency
distribution of functions in the scaled kernel space.

To avoid this loss, we work directly with the Fourier representation of kernel network 
$$
v=\sum_{j=1}^N a_j\Phi_\delta(\cdot-x_j) \in V_{\Phi_\delta,X}.
$$
Its Fourier transform is of the form
$$
\widehat v(\omega)
=
\widehat\Phi(\delta\omega)P_X(\omega),
\quad
P_X(\omega)
=
\sum_{j=1}^N a_j e^{-ix_j^T\omega}.
$$
Consequently, for $s\ge0$,
$$
\|v\|_{H^s(\mathbb R^d)}^2
=
(2\pi)^{-d/2}\int_{\mathbb R^d}(1+\|\omega\|_2^2)^s|\widehat\Phi(\delta\omega)|^2|P_X(\omega)|^2\,d\omega.
$$
The important observation is that the same kernel multiplier
$|\widehat\Phi(\delta\omega)|^2$ occurs in both the strong and weak Sobolev norms.  Keeping this multiplier inside the argument allows its $\delta$-dependence to be compared at the same frequencies,
rather than estimating the two norms independently.
The natural frequency threshold is of order $q_X^{-1}$.  We
therefore split the Fourier domain into
$$
\|\omega\|_2\leq c\, q_X^{-1}
\quad\text{and}\quad
\|\omega\|_2\geq c\, q_X^{-1}.
$$
On the low-frequency region, the desired inverse factor follows
directly from
$$
(1+\|\omega\|_2^2)^{\sigma-\alpha}
\leq c\,
q_X^{-2(\sigma-\alpha)}.
$$
No estimate of the kernel multiplier $|\widehat\Phi(\delta\omega)|^2$ is needed there. In the high frequency region the analysis is more detailed. We develop this observation into a precise theoretical result in the following.

To control the exponential polynomial $P_X$ uniformly with respect to
the number and location of the centers, we use an Ingham-type frame
estimate for separated exponentials. See, for example, the theory of
multivariate Fourier series and Ingham inequalities
\cite{KunisMollerPeterVonDerOhe2018,PottsTasche2013}.
Related stability results appear in the kernel approximation
literature through frame and coefficient-stability estimates for
spherical basis function networks. See
\cite{MhaskarNarcowichPrestinWard2000, MhaskarNarcowichPrestinWard2010, MhaskarNarcowichWard2001, NarcowichPetrushevWard2006}.
Although the setting and the underlying frame systems are different,
these estimates similarly relate geometric separation of the centers
to stable control of finite-dimensional function spaces.
For completeness, and because only a non-sharp version is needed here,
we include a proof based on a smooth Fourier cutoff and a packing
argument.

\begin{lemma}\label{lem:exponential-frame}
Let
$
X=\{x_1,\ldots,x_N\}\subset\mathbb R^d
$
be a finite set with separation distance $q_X$,
and let
$$
P_X(\omega) = \sum_{j=1}^N a_j e^{-i x_j^T \omega},
\quad \boldsymbol a=(a_1,\ldots,a_N)^T\in\mathbb C^N.
$$
Then there exist constants
$\kappa_0>0$ and 
$0<c_{\rm P}\le C_{\rm P}<\infty$
depending only on $d$, such that, whenever
$
Rq_X\ge\kappa_0,
$
we have
\begin{equation}\label{eq:exp-frame-ball}
c_{\rm P}R^d\|\boldsymbol a\|_{\ell_2}^2
\le
\int_{B(0,R)}|P_X(\omega)|^2\,d\omega
\le
C_{\rm P}R^d\|\boldsymbol a\|_{\ell_2}^2.
\end{equation}
Moreover, there exist fixed constants
$\kappa_1>0$
and
$\lambda>1$
depending only on $d$, such that, with
$$
R_1:=\frac{\kappa_1}{q_X},
\quad
R_2:=\lambda R_1,
$$
we have
\begin{equation}\label{eq:exp-frame-annulus}
\int_{R_1<\|\omega\|_2\le R_2} |P_X(\omega)|^2\,d\omega
\ge
c_{\rm A}R_1^d \|\boldsymbol a\|_{\ell_2}^2,
\end{equation}
where $c_{\rm A}>0$ depends only on $d$.
\end{lemma}

\begin{proof}
We first establish the estimates on balls.
Choose a real-valued, nonnegative, even function
$$
\varphi\in C_0^\infty(B(0,1))
$$
such that $\varphi\not\equiv0$. With the Fourier transform convention
used throughout the paper, expanding the square gives
\begin{equation}
\begin{aligned}\label{eq:exp-chi-expansion}
\int_{\mathbb R^d}
\varphi(\omega/R)|P_X(\omega)|^2\,d\omega
&=
\sum_{j,k=1}^N
a_j\overline{a_k}
\int_{\mathbb R^d}
\varphi(\omega/R)
e^{-i(x_j-x_k)^T \omega}\,d\omega
\\
&=
(2\pi)^{d/2} R^d
\sum_{j,k=1}^N
a_j\overline{a_k}
\widehat\varphi\bigl(R(x_j-x_k)\bigr).
\end{aligned}
\end{equation}
Since $\varphi\in C_0^\infty(\mathbb R^d)$,
its Fourier transform is rapidly decreasing. Hence, for every
integer $M>0$, there exists $C_M>0$ such that
\begin{equation}
\label{eq:chi-hat-decay}
|\widehat\varphi(\xi)|
\le
C_M(1+\|\xi\|_2)^{-M},
\quad
\xi\in\mathbb R^d.
\end{equation}
We shall use the following elementary consequence of the
separation of $X$. For any $M>d$,
\begin{equation}
\label{eq:packing-sum}
\sup_{1\le j\le N}
\sum_{\substack{k=1\\k\neq j}}^N
\bigl(1+R\|x_j-x_k\|_2\bigr)^{-M}
\le
C_{d,M}(Rq_X)^{-M+d},
\quad
Rq_X\ge1.
\end{equation}
Indeed, the balls $B(x_k,q_X)$ are pairwise disjoint. Dividing
the points relative to $x_j$ into annuli
$$
2mq_X \le \|x_k-x_j\|_2 < 2(m+1)q_X, \quad m=1,2,\ldots,
$$
a standard packing argument shows that the number of centers in
the $m$th annulus is bounded by
$C_d(1+m)^{d-1}$.
Therefore
$$
\begin{aligned}
\sum_{k\neq j}\bigl(1+R\|x_j-x_k\|_2\bigr)^{-M}
&\le C_d \sum_{m=1}^\infty m^{d-1}(1+2mRq_X)^{-M}\\
&\le C_{d,M} (Rq_X)^{-M} \sum_{m=1}^\infty m^{d-1-M}.
\end{aligned}
$$
Since $M>d$, the series converges. This proves, in fact, the
slightly stronger estimate
$C(Rq_X)^{-M}$,
and hence \eqref{eq:packing-sum}. We now return to \eqref{eq:exp-chi-expansion}. The diagonal terms
are
$$
(2\pi)^{d/2}R^d\widehat\varphi(0) \|\boldsymbol a\|_{\ell_2}^2,
$$
where
$\widehat\varphi(0)>0$.
For the off-diagonal terms, using $2|a_j||a_k|\le |a_j|^2+|a_k|^2$
together with \eqref{eq:chi-hat-decay} and
\eqref{eq:packing-sum}, gives
\begin{align*}
\left|\sum_{j\neq k}a_j\overline{a_k}\widehat\varphi\bigl(R(x_j-x_k)\bigr)\right|
&\le
C_M \sum_{j=1}^N |a_j|^2 \sum_{k\neq j} \bigl(1+R\|x_j-x_k\|_2\bigr)^{-M}\\
&\le
C(Rq_X)^{-M+d} \|\boldsymbol a\|_{\ell_2}^2.
\end{align*}
Thus, by choosing $\kappa_0>0$ sufficiently large, we ensure that
for $Rq_X\ge\kappa_0$ the off-diagonal contribution is at most
one half of the diagonal contribution. Consequently,
\begin{equation*}
\int_{\mathbb R^d}\varphi(\omega/R)|P_X(\omega)|^2\,d\omega
\ge
cR^d\|\boldsymbol a\|_{\ell_2}^2.
\end{equation*}
Since $\varphi$ is supported in $B(0,1)$,
$$
\int_{\mathbb R^d}
\varphi(\omega/R)|P_X(\omega)|^2\,d\omega
\le
\|\varphi\|_{L^\infty}
\int_{B(0,R)}
|P_X(\omega)|^2\,d\omega.
$$
This proves the lower estimate in
\eqref{eq:exp-frame-ball}.

To obtain the upper estimate, choose an even nonnegative function
$$
\widetilde\varphi\in C_0^\infty(B(0,2))
$$
such that
$\widetilde\varphi(\omega)\ge1$
for $\omega\in B(0,1)$. Then
$$
\int_{B(0,R)}|P_X(\omega)|^2\,d\omega
\le
\int_{\mathbb R^d} \widetilde\varphi(\omega/R)|P_X(\omega)|^2\,d\omega.
$$
Repeating the calculation above, but now taking absolute values
of both diagonal and off-diagonal terms, gives
$$
\int_{\mathbb R^d} \widetilde\varphi(\omega/R)|P_X(\omega)|^2\,d\omega
\le C R^d\|\boldsymbol a\|_{\ell_2}^2
$$
whenever $Rq_X\ge\kappa_0$, after increasing $\kappa_0$ if
necessary. This proves the upper estimate in
\eqref{eq:exp-frame-ball}.

It remains to obtain the annular lower bound. Choose
$\kappa_1\ge\kappa_0$
and set
$R_1={\kappa_1}/{q_X}$
and let
$R_2=\lambda R_1$
where $\lambda>1$ will be fixed below. By
\eqref{eq:exp-frame-ball},
$$
\int_{B(0,R_2)}|P_X(\omega)|^2\,d\omega
\ge
c_{\rm P}R_2^d\|\boldsymbol a\|_{\ell_2}^2,
$$
whereas
$$
\int_{B(0,R_1)}|P_X(\omega)|^2\,d\omega
\le
C_{\rm P}R_1^d\|\boldsymbol a\|_{\ell_2}^2.
$$
Therefore
\begin{align*}
\int_{R_1<\|\omega\|_2\le R_2}|P_X(\omega)|^2\,d\omega
&=
\int_{B(0,R_2)}|P_X|^2\,d\omega-\int_{B(0,R_1)}|P_X|^2\,d\omega\\
&\ge
\left(c_{\rm P}\lambda^d-C_{\rm P}\right)R_1^d\|\boldsymbol a\|_{\ell_2}^2.
\end{align*}
Choose $\lambda>1$ so large that
$c_{\rm P}\lambda^d-C_{\rm P}>0$.
For example, it is enough to have
$\lambda^d\ge{2C_{\rm P}}/{c_{\rm P}}$.
Then
$c_{\rm P}\lambda^d-C_{\rm P}\ge C_{\rm P}>0$,
and \eqref{eq:exp-frame-annulus} follows.
\end{proof}

\begin{theorem}
\label{thm:scale-uniform-whole-space}
Let $\Phi:\mathbb R^d\to\mathbb R$ be a strictly positive
definite kernel satisfying the decay condition \eqref{eq:fourier_decay}
for some $\sigma>d/2$. 
Fix $c_0>0$. Then, for every $0\leq\alpha\leq \tau\leq\sigma$
there exists a constant $C=C_{d,\sigma,\tau,\alpha,\Phi,c_0}>0$
such that
\begin{equation}\label{eq:scale-uniform-whole-space-Bernstein}
\|v\|_{H^\tau(\mathbb R^d)}
\le
Cq_X^{\alpha-\tau}
\|v\|_{H^\alpha(\mathbb R^d)}
\end{equation}
for every $0<\delta\le1$, every finite $X=\{x_1,\ldots,x_N\}\subset \R^d$ satisfying $q_X\le c_0\delta$, and every $v\in V_{\Phi_\delta,X}$.
The constant $C$ is in particular independent of $\delta$, $X$, $N$, and $v$.
\end{theorem}

\begin{proof}
First we prove the result for the endpoint case $\tau = \sigma$.
The case $\alpha=\sigma$ is trivial. We therefore assume
$0\le\alpha<\sigma$.
Let
$$
v(x)=\sum_{j=1}^N a_j\Phi_\delta(x-x_j).
$$
Then we have 
\begin{equation*}
\widehat v(\omega)=\widehat\Phi(\delta\omega)P_X(\omega), \quad P_X(\omega)
= \sum_{j=1}^N a_j e^{-ix_j^T\omega},
\end{equation*}
Thus, for $0\le s\le\sigma$,
\begin{equation}\label{eq:kernel-network-Hs}
\|v\|_{H^s(\mathbb R^d)}^2
=(2\pi)^{-d/2}\int_{\mathbb R^d}(1+\|\omega\|_2^2)^s|\widehat\Phi(\delta\omega)|^2
|P_X(\omega)|^2\,d\omega,
\end{equation}
up to a fixed Fourier-normalization constant.
Let $\kappa_1$ and $\lambda$ be the constants from
Lemma~\ref{lem:exponential-frame}. Increasing $\kappa_1$ if
necessary, we assume in addition that
\begin{equation*}
\label{eq:kappa-choices}
\frac{\kappa_1}{c_0}\ge1,
\quad
\frac{\kappa_1}{c_0}\ge\lambda^{-1},
\quad
\kappa_1\ge c_0.
\end{equation*}
Define
$$
R_1:=\frac{\kappa_1}{q_X},
\quad
R_2:=\lambda R_1.
$$
Since $q_X\le c_0\delta\le c_0$, we have
$R_1\ge {\kappa_1}/{c_0}\ge1$.
Split
\begin{equation}
\label{eq:Hsigma-split}
\|v\|_{H^\sigma(\mathbb R^d)}^2
=
I_{\rm low}+I_{\rm high},
\end{equation}
where
$$
I_{\rm low}
=
\int_{\|\omega\|_2\le R_1}
(1+\|\omega\|_2^2)^\sigma
|\widehat v(\omega)|^2\,d\omega
$$
and
$$
I_{\rm high}
=
\int_{\|\omega\|_2>R_1}
(1+\|\omega\|_2^2)^\sigma
|\widehat v(\omega)|^2\,d\omega.
$$
We now split the remaining parts of the proof into three steps to find an upper bound for $I_{\rm low}$ (Step 1), and upper bound for $I_{\rm high}$ (Step 2), and a lower bound for the weaker Sobolev norm $\|v\|_{H^\alpha(\R^d)}$ (Step 3).

\medskip
\noindent
{Step 1:}
For $\|\omega\|_2\le R_1$, we have 
$$
(1+\|\omega\|_2^2)^\sigma
\le
(1+R_1^2)^\beta
(1+\|\omega\|_2^2)^\alpha.
$$
where $\beta:=\sigma-\alpha>0$.
Therefore
$$
I_{\rm low}
\le
(1+R_1^2)^\beta
\|v\|_{H^\alpha(\mathbb R^d)}^2,
$$
Since
$R_1={\kappa_1}/{q_X}$ and $q_X\le c_0$, we can write 
$$
1+R_1^2
=
1+\frac{\kappa_1^2}{q_X^2}
\le
\frac{c_0^2+\kappa_1^2}{q_X^2}.
$$
Hence
\begin{equation}
\label{eq:Ilow-final}
I_{\rm low}
\le
Cq_X^{-2(\sigma-\alpha)}
\|v\|_{H^\alpha(\mathbb R^d)}^2.
\end{equation}

\medskip
\noindent
{Step 2:} 
If $\|\omega\|_2>R_1$, then from condition $q_X\leq c_0\delta$ we have 
$$
\delta\|\omega\|_2>\delta R_1=\frac{\kappa_1\delta}{q_X}\ge\frac{\kappa_1}{c_0}\ge1.
$$
Thus the algebraic decay assumption \eqref{eq:scaled-fourier-decay} implies
\begin{equation*}
|\widehat\Phi(\delta\omega)|^2
\le
C\delta^{-4\sigma}\|\omega\|_2^{-4\sigma}.
\end{equation*}
Also, because of $\|\omega\|_2\ge R_1\ge1$, we have
$$
(1+\|\omega\|_2^2)^\sigma
\le
C\|\omega\|_2^{2\sigma}.
$$
Consequently,
\begin{equation}
\label{eq:Ihigh-exp}
I_{\rm high}
\le
C\delta^{-4\sigma}
\int_{\|\omega\|_2>R_1}
\|\omega\|_2^{-2\sigma}
|P_X(\omega)|^2\,d\omega.
\end{equation}
In order to obtain a proper bound for the right-hand side of \eqref{eq:Ihigh-exp}, decompose the exterior region $\|\omega\|_2>R_1$ into the annuli
$$
A_k=\left\{2^kR_1<\|\omega\|_2\le2^{k+1}R_1\right\},
\quad k=0,1,\ldots .
$$
By the upper ball estimate
\eqref{eq:exp-frame-ball}, we have 
$$
\int_{B(0,2^{k+1}R_1)}
|P_X(\omega)|^2\,d\omega
\le
C(2^{k+1}R_1)^d
\|\boldsymbol a\|_{\ell_2}^2.
$$
Thus
$$
\begin{aligned}
\int_{A_k}
\|\omega\|_2^{-2\sigma}|P_X(\omega)|^2\,d\omega
&\leq
(2^kR_1)^{-2\sigma}
\int_{B(0,2^{k+1}R_1)}
|P_X(\omega)|^2\,d\omega
\\
&\leq C 2^{-k(2\sigma-d)} R_1^{d-2\sigma} \|\boldsymbol a\|_{\ell_2}^2.
\end{aligned}
$$
This shows that
$$
\int_{\|\omega\|_2>R_1} \|\omega\|_2^{-2\sigma} |P_X(\omega)|^2\,d\omega \leq  C R_1^{d-2\sigma}
\|\boldsymbol a\|_{\ell_2}^2 \sum_{k=0}^\infty 2^{-k(2\sigma-d)}.
$$
Because
$
2\sigma-d>0,
$
the geometric series converges. Therefore
\begin{equation}
\label{eq:weighted-exp-tail}
\int_{\|\omega\|_2>R_1} \|\omega\|_2^{-2\sigma} |P_X(\omega)|^2\,d\omega \leq C R_1^{d-2\sigma}
\|\boldsymbol a\|_{\ell_2}^2.
\end{equation}
Combining
\eqref{eq:Ihigh-exp} and
\eqref{eq:weighted-exp-tail}, we obtain
\begin{equation}
\label{eq:Ihigh-coeff}
I_{\rm high}
\le
C
\delta^{-4\sigma}
R_1^{d-2\sigma}
\|\boldsymbol a\|_{\ell_2}^2.
\end{equation}

\medskip
\noindent
{Step 3:} Consider the fixed-ratio annulus
$$
A_{12}= \left\{R_1<\|\omega\|_2\le R_2 \right\}.
$$
By Lemma~\ref{lem:exponential-frame},
\begin{equation}
\label{eq:annulus-use}
\int_{A_{12}} |P_X(\omega)|^2\,d\omega \geq c_{\rm A}R_1^d \|\boldsymbol a\|_{\ell_2}^2.
\end{equation}
For $\omega\in A_{12}$, we therefore have 
$R_1\le\|\omega\|_2\le\lambda R_1$.
Since $R_1\ge1$,
\begin{equation}
\label{eq:Halpha-weight-annulus}
(1+\|\omega\|_2^2)^\alpha
\ge
cR_1^{2\alpha}.
\end{equation}
Furthermore,
$\delta\|\omega\|_2 \geq \delta R_1 \geq {\kappa_1}/{c_0}\ge1$.
On the other hand,
$\delta\|\omega\|_2\leq \lambda\delta R_1$.
Using the lower Fourier-decay bound in
\eqref{eq:scaled-fourier-decay},
$
\widehat\Phi(\delta\omega)\ge c_\Phi (1+\delta^2\|\omega\|_2^2)^{-\sigma}.
$
Since $\delta\|\omega\|_2\ge1$,
$$
1+\delta^2\|\omega\|_2^2 \leq 2\delta^2\|\omega\|_2^2
\leq 2\lambda^2\delta^2 R_1^2.
$$
Therefore
\begin{equation}
\label{eq:kernel-high-lower}
|\widehat\Phi(\delta\omega)|^2
\ge
c\, \delta^{-4\sigma}R_1^{-4\sigma},
\quad
\omega\in A_{12},
\end{equation}
where the constant $c$ depends on $\lambda$, $\sigma$, and $c_\Phi$,
but not on $\delta$ or $X$.
It follows from
\eqref{eq:kernel-network-Hs},
\eqref{eq:Halpha-weight-annulus},
\eqref{eq:kernel-high-lower}, and
\eqref{eq:annulus-use} that
$$
\begin{aligned}
\|v\|_{H^\alpha(\mathbb R^d)}^2
&\geq
\int_{A_{12}}
(1+\|\omega\|_2^2)^\alpha
|\widehat\Phi(\delta\omega)|^2
|P_X(\omega)|^2\,d\omega
\\
&\geq
c\,
R_1^{2\alpha}
\delta^{-4\sigma}
R_1^{-4\sigma}
\int_{A_{12}}|P_X(\omega)|^2\,d\omega
\\
&\ge
c\,
\delta^{-4\sigma}
R_1^{d+2\alpha-4\sigma}
\|\boldsymbol a\|_{\ell_2}^2.
\end{aligned}
$$
Hence
\begin{equation}
\label{eq:coeff-via-Halpha}
\delta^{-4\sigma} \|\boldsymbol a\|_{\ell_2}^2
\le C R_1^{4\sigma-2\alpha-d} \|v\|_{H^\alpha(\mathbb R^d)}^2.
\end{equation}
Substituting \eqref{eq:coeff-via-Halpha} into
\eqref{eq:Ihigh-coeff} gives
$$
\begin{aligned}
I_{\rm high} &\le C R_1^{d-2\sigma} R_1^{4\sigma-2\alpha-d} \|v\|_{H^\alpha(\mathbb R^d)}^2 \\
&= C R_1^{2(\sigma-\alpha)} \|v\|_{H^\alpha(\mathbb R^d)}^2.
\end{aligned}
$$
Since
$R_1={\kappa_1}/{q_X}$,
we conclude that
\begin{equation}
\label{eq:Ihigh-final}
I_{\rm high}
\le
C
q_X^{-2(\sigma-\alpha)}
\|v\|_{H^\alpha(\mathbb R^d)}^2.
\end{equation}
Finally, combining \eqref{eq:Hsigma-split},
\eqref{eq:Ilow-final}, and \eqref{eq:Ihigh-final}, and taking square roots, we obtain
$$
\|v\|_{H^\sigma(\mathbb R^d)} \le C q_X^{\alpha-\sigma} \|v\|_{H^\alpha(\mathbb R^d)},
$$
which proves \eqref{eq:scale-uniform-whole-space-Bernstein} for $\tau=\sigma$. 
The constant depends only on the fixed quantities
$d,\sigma,\alpha,\Phi$, and $c_0$, and are independent of
$\delta$, $X$, $N$, and $v$.

\medskip
\noindent
Finally, we prove the theorem for the case $0\leq\tau < \sigma$ using the interpolation theory in Sobolev spaces. 
Let
$0\le\alpha<\tau<\sigma$, and choose $\theta\in(0,1]$ such that
$$
\tau=(1-\theta)\alpha+\theta\sigma,
\quad
\theta=\frac{\tau-\alpha}{\sigma-\alpha}.
$$
We use the interpolation scale in Sobolev spaces. See for example \cite{BerghLofstrom1976,Triebel1978}. We have
$$
\|v\|_{H^\tau(\mathbb R^d)}
\leq
\|v\|_{H^\sigma(\mathbb R^d)}^\theta
\|v\|_{H^\alpha(\mathbb R^d)}^{1-\theta}
$$
We now apply the above inverse inequality to get
$$
\begin{aligned}
\|v\|_{H^\tau(\mathbb R^d)}
&\leq \left(Cq_X^{\alpha-\sigma}\|v\|_{H^\alpha(\mathbb R^d)}\right)^\theta \|v\|_{H^\alpha(\mathbb R^d)}^{1-\theta}  \\
& = C^\theta q_X^{\theta(\alpha-\sigma)} \|v\|_{H^\alpha(\mathbb R^d)}\\
&= C^\theta q_X^{\alpha-\tau} \|v\|_{H^\alpha(\mathbb R^d)},
\end{aligned}
$$
where we used $\theta(\alpha-\sigma)=\alpha-\tau$. Absorbing $C^\theta$ into the constant proves
\eqref{eq:scale-uniform-whole-space-Bernstein}.
\end{proof}

\section{Concluding remarks}

In this paper, we have developed inverse inequalities for scaled kernel
spaces with particular attention to the dependence of the estimates on
the kernel scale $\delta$. One of the main results, given in
Corollary~\ref{cor:scale-uniform-bernstein}, is the bounded-domain
estimate
\begin{equation*}
\|v\|_{H^\tau(\Omega)} \leq C q_X^{-\tau}\|v\|_{L^2(\Omega)},
\quad v\in V_{\Phi_\delta,X}, \quad 0\leq\tau\leq\sigma,
\end{equation*}
where $\sigma$ represents the Sobolev smoothness associated with the
reference kernel $\Phi$. Most importantly, the constant $C$ is
independent of both the scale parameter $\delta$ and the center set
$X$. The estimate requires only the one-sided condition
$$
q_X\leq c_0\delta.
$$
Thus, the kernel scale is allowed to decrease more slowly than the
separation distance, and in particular the ratio $q_X/\delta$ may tend
to zero under refinement. This flexibility is important in multiscale
kernel methods, where the kernel scale and the discretization parameters
need not decrease at the same rate.
The scaled native-space and band-limited interpolation arguments used
to obtain this result led to the $L^2(\Omega)$ norm on the
weaker side. A direct attempt to use the same idea for a positive
Sobolev index $\alpha$, even if we replace $\Omega$ by $\R^d$, introduces additional powers of $\delta$ and
therefore does not yield the desired scale-uniform estimate. To overcome
this difficulty on the whole space, we developed a different argument
that works directly with the Fourier representation of functions in
$V_{\Phi_\delta,X}$. By separating the Fourier domain into low and
high frequency regions and using frame estimates for the exponential
polynomial associated with the center set, we obtained the more general
scale-uniform inequality
\begin{equation*}
\|v\|_{H^\tau(\mathbb R^d)}
\leq C q_X^{\alpha-\tau}\|v\|_{H^\alpha(\mathbb R^d)},
\quad v\in V_{\Phi_\delta,X}, \quad 0\leq\alpha\leq\tau\leq\sigma,
\end{equation*}
again under the condition $q_X\leq c_0\delta$. In particular, no
additional power of $\delta$ remains in the final estimate.

A natural next question is whether the latter result can be extended
to bounded domains while having both the full range
$0\leq\alpha\leq\tau\leq\sigma$ and the scale condition
$q_X\leq c_0\delta$. This does not follow by simply restricting the
whole-space inequality to $\Omega$, since the weaker whole-space norm
cannot in general be controlled by the corresponding norm of the
restriction to $\Omega$. The boundary therefore introduces a genuine
difficulty that is absent from the whole-space Fourier analysis.
For fixed-scale kernel spaces, Cheung, Ling, and Schaback
\cite{CheungLingSchaback2018} address the bounded-domain problem by
introducing an auxiliary kernel associated with a differential
operator and combining a band-limited interpolation argument with the
optimality property of the kernel interpolant. A direct adaptation of
this approach to the scaled spaces generated by $\Phi_\delta$ produces
an additional negative power of $\delta$ in an intermediate estimate (optimality of the interpolant),
which prevents uniformity as $\delta\to0$. We also found that a
coefficient-stability approach based on frame estimates, which we used here for the whole space estimates, can recover
the desired bounded-domain inequality for
compactly supported kernels, but only under the stronger
two-sided condition
$c_1\delta\leq q_X\leq c_2\delta$.
This excludes the multiscale regime $q_X/\delta\to0$ that is of primary interest here. For this reason, we
have not pursued that restricted result in the present paper.
Establishing a general bounded-domain inequality of the form
$$
\|v\|_{H^\tau(\Omega)} \leq C q_X^{\alpha-\tau} \|v\|_{H^\alpha(\Omega)}, \quad 0\leq\alpha\leq\tau\leq\sigma,
$$
with a constant uniform in $\delta$ under only $q_X\leq c_0\delta$, therefore remains an open problem.

\section*{Acknowledgments}
In preparing this paper, AI tools were used to assist with the presentation and refinement of the text, particularly for language editing and improving the exposition of all sections. They also helped in sketching and refining some proofs by suggesting alternative directions. All mathematical arguments and suggestions were independently checked by the author. The mathematical content, underlying concepts, organization of the material, and accuracy of the results remain the sole responsibility of the author.

\bibliographystyle{plain}
\bibliography{refs}

\end{document}